\documentclass[12pt]{article}
\usepackage[utf8]{inputenc}
\usepackage{amsthm, amssymb}
\usepackage{mathtools}
\usepackage{geometry}
\usepackage{graphicx}
\usepackage{amsmath, enumerate, bbm}
\usepackage{tikz}
\usepackage{url}
\usetikzlibrary{matrix}

\newcommand{\bN}{\mathbb{N}}

\newcommand{\bT}{\mathbb{T}}
\newcommand{\bZ}{\mathbb{Z}}

\newtheorem{thm}{Theorem}[section]
\newtheorem{lem}[thm]{Lemma}

\newtheorem{cor}[thm]{Corollary}

\newtheorem{claim}[thm]{Claim}

\theoremstyle{definition}

\newtheorem{qu}{Question}

\def\Pre{\mathrm{Pre}}
\def\IW{\mathrm{IW}}

\begin{document}

\title{Smallest percolating sets \\ in bootstrap percolation on grids}
\author{Micha{\l} Przykucki\thanks{School of Mathematics, University of Birmingham, Edgbaston, Birmingham B15 2TT, United Kingdom. \newline  E-mail: \texttt{m.j.przykucki@bham.ac.uk, tomshelton148@aol.com.}} \thanks{Supported by the EPSRC grant EP/P026729/1.} \and Thomas Shelton\footnotemark[1]}
\maketitle

\begin{abstract}
In this paper we fill in a fundamental gap in the extremal bootstrap percolation literature, by providing the first proof of the fact that for all $d \geq 1$, the size of the smallest percolating sets in $d$-neighbour bootstrap percolation on $[n]^d$, the $d$-dimensional grid of size $n$, is $n^{d-1}$. Additionally, we prove that such sets percolate in time at most $c_d n^2$, for some constant $c_d >0 $ depending on $d$ only.
\end{abstract}

\section{Introduction}
\label{sec:intro}

\emph{Bootstrap percolation}, suggested by Chalupa, Leath, and Reich~\cite{bootstrapbethe}, is a simple cellular automaton modelling the spread of an infection on the vertex set of a graph $G$. For some positive integer $r$, given a set of initially infected vertices $A \subseteq V(G)$, in consecutive rounds we infect all vertices with at least $r$ already infected neighbours. \emph{Percolation} occurs if every vertex of $G$ is eventually infected.

The majority of research into bootstrap percolation processes has been focused on the probabilistic properties of the model. More precisely, if we initially infect every vertex independently at random with some probability $p$, how likely is the system to percolate? The monotonicity of the model (i.e., the fact that infected vertices never heal) makes it reasonable to ask about the value of the \emph{critical probability}~$p$, above which percolation becomes more likely to occur than not. This quantity has been analysed for many different families of graphs $G$ and for various infection rules, and often very sharp results have been obtained by, e.g., Aizenman and Lebowitz~\cite{metastabilityeffects}, Holroyd~\cite{sharpmetastability}, and Balogh, Bollob{\'a}s, Duminil-Copin, and Morris~\cite{sharpbootstrapall}.

Another family of questions related to bootstrap percolation that have been studied is concerned with the extremal properties of the model. Morris~\cite{largestgridbootstrap} analysed the size of the largest minimal percolating sets in $2$-neighbour bootstrap percolation on the $n \times n$ square. For the same setup, Benevides and Przykucki~\cite{maxtime} determined the maximum time the process can take until it stabilises. However, the first extremal question that attracted attention in bootstrap percolation was about the size of the smallest percolating sets. For grid graphs, this has been studied by Pete~\cite{diseaseprocesses} (the summary of Pete's results can be found in Balogh and Pete~\cite{randomdisease}). For the hypercube, the size of the smallest percolating sets for all values of the infection threshold was found by Morrison and Noel~\cite{extremalcube}. Feige, Krivelevich, and Reichman \cite{contagiousGnp} analysed the size of these sets in random graphs, while Coja-Oghlan, Feige, Krivelevich, and Reichman \cite{contagiousExpanders} studied such sets in expander graphs.

\subsection{The $d$-neighbour process in $d$ dimensions}

Let us introduce some notation. For $n \in \bN$, let $[n] = \{1,2, \ldots, n\}$. The $d$-dimensional grid graph of size $n$ is the graph with vertex set $[n]^d$, in which $u,v \in [n]^d$ are adjacent if and only if they differ by a value of 1 in exactly one coordinate. For $d,r,n \in \bN$, let $G_{d,r}(n)$ denote the size of the smallest percolating sets in $r$-neighbour bootstrap percolation on $[n]^d$. For a set $A \subset [n]^d$, let $\langle A \rangle_r$ be the \emph{closure} of $A$ in $r$-neighbour bootstrap percolation, i.e., the set of all vertices that become infected in the process that was started from $A$.

Among the results stated in~\cite{diseaseprocesses} (see also the Perimeter Lemma in the Appendix to~\cite{randomdisease}) is the following theorem.
\begin{thm}
\label{thm:pete}
For all $n, d \in \bN$, we have $G_{d,d}(n)=n^{d-1}$.
\end{thm}

This is obviously trivial for $d=1$, and the case when $d=2$ constitutes a lovely and well-known puzzle. Indeed, finding a percolating set of size $n$ is easy: just take one of the diagonals of the square. To show that there is no percolating set of size strictly less than $n$, we can refer to the famous \emph{perimeter argument}: the perimeter of the infected set (understood as the number of edges between an infected and a healthy vertex, if we naturally embed our square $[n]^2$ in the infinite grid $\bZ^2$) can never grow. Indeed, whenever a new vertex becomes infected, it is by virtue of at least two perimeter edges. Thus at least two edges are removed from the perimeter of the infected set, and at most two new ones are added, and the aforementioned monotonicity of the perimeter follows. Since the whole $n \times n$ grid has perimeter $4n$, and any initially infected vertex contributes at most $4$ edges to the perimeter, we need at least $n$ initially infected vertices to percolate.

Somewhat surprisingly, the perimeter argument carries immediately to higher dimensions, giving us the appropriate lower bound $G_{d,d}(n) \geq n^{d-1}$ for all $d \in \bN$. As for the upper bound, there is a natural candidate, sometimes referred to as a ``cyclic combination'' of the one-dimensional lower set. More precisely, for $d\leq k\leq dn$, let $V_k=\{v=(v_1,...,v_d)\in [n]^d:\sum_{i=1}^{d} v_i=k\}$. It is then natural to believe that the set
\begin{equation}
\label{eqn:initialSet}
A = A_d = \bigcup_{i=1}^{d} V_{in}
\end{equation}
percolates in $d$-neighbour bootstrap percolation on $[n]^d$, and indeed this is the construction that was used to deduce the upper bound in~\cite{diseaseprocesses}. One can imagine how two ``neighbouring hyperplanes'', $V_{(i-1)n}$ and $V_{in}$, fill in the space between them with infection until the two growths meet, from which point on the process quickly finishes. The fact that $G_{d,d}(n) = n^{d-1}$ has become a ``folklore knowledge'' in the area of bootstrap percolation, and has sometimes even been referred to as an ``observation''. Up to our best knowledge~\cite{PetePrivate}, no formal proof of Theorem~\ref{thm:pete} was provided in~\cite{diseaseprocesses}, and no such proof exists in the literature.

However, problems arise quickly when one tries to describe how exactly the space between the two hyperplanes is filled in. Any vertex in $V_{(i-1)n+1}$ with at least one coordinate equal to $1$ has fewer than $d$ infected neighbours in $V_{(i-1)n}$, and consequently does not become infected in step~1. Similarily, after one step, any vertex in $V_{(i-1)n+2}$ with at least one coordinate equal to at most $2$ has fewer than $d$ infected neighbours in $V_{(i-1)n+1}$, and also remains healthy. This problem builds up (analogous constraints can be easily formulated for the layers being infected ``from above'' by $V_{in}$) and, in fact, the two growths barely meet - two hyperplanes at distance $n+1$ apart would have stayed separated, while hyperplanes at distance $n-1$ would result in some vertices being infected by more than $d$ infected neighbours, and consequently no percolation by the perimeter argument.

What is however even more troublesome, describing the growth from the moment of the meeting onwards is where the real challenges occur. By the perimeter argument, we know that we have no elbow room in this description: no proper subset of $A$ percolates, and even a small perturbation of $A$ would not percolate if any vertex ever became infected by virtue of more than $d$ infected neighbours. In Figure~\ref{fig:G_{3,3}(6)} we present the growth of the infected set, starting from $A$ as defined in~\eqref{eqn:initialSet}, in $3$-neighbour bootstrap percolation on $[6]^3$. Even though we are in just three dimensions, and the size of the grid is very small, the process already feels quite difficult to describe and lasts as many as 14 steps. Consequently, we believe that Theorem~\ref{thm:pete} requires a proper, formal proof, which we provide as the main result of this paper in Section~\ref{sec:main}.

\begin{figure}[htb]
    \centering
    \includegraphics[scale=0.41]{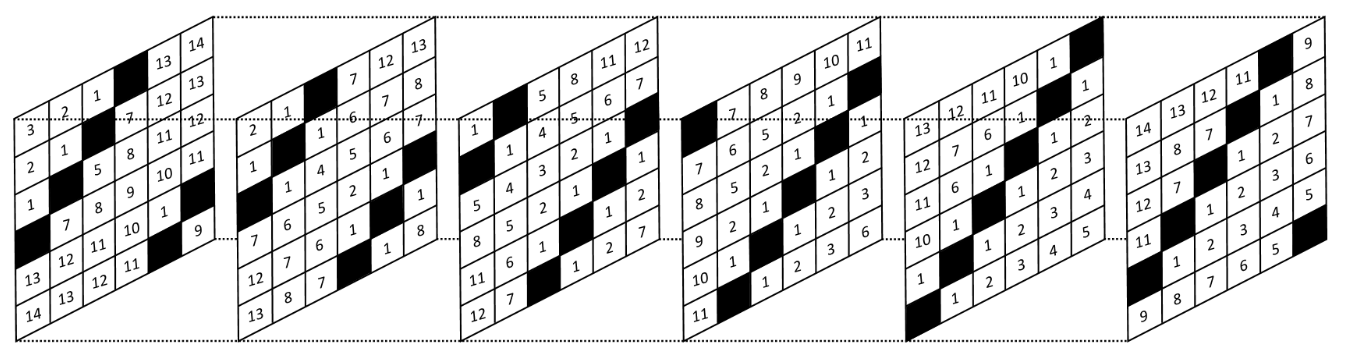}
    \caption{Example showing the spread of infection in $[6]^3$ starting from a set of size $G_{3,3}(6) = 6^2 = 36$.}
    \label{fig:G_{3,3}(6)}
\end{figure}

Another reason to convince oneself about the fact that the process of filling in the space between $V_{(i-1)n}$ and $V_{in}$ is nontrivial becomes apparent when we look at the results of numerical simulations, and analyse the time the process takes to terminate. It quickly becomes apparent that, for a fixed $d$, this time grows quadratically with $n$. This should be somewhat surprising, as by averaging there is some $i$ such that the volume between $V_{(i-1)n}$ and $V_{in}$ is of the order $n^d$. For a percolating set $A$, let $T(A)$ be the time (i.e., the number of time steps) the process takes to infect the whole vertex set. Let
\begin{equation}
\label{eqn:smallestTimeDefn}
m_d(n) = \min \left \{T(A) : \langle A \rangle_d = [n]^d, |A| = n^{d-1} \right \}.
\end{equation}
In Section~\ref{sec:time}, we come back to the question of percolation time and we prove the following theorem.
\begin{thm}
\label{thm:smallestTime}
 We have $m_1(n) = \lceil n/2 \rceil$, $m_2(n)  = n-1$, and for $d \geq 3$,
 \[
  \frac{dn}{2}+O(1) \leq m_d(n) \leq (d+2)n^2+n.
 \]
\end{thm}

Before we proceed to the main part of this work, let us emphasise the importance of the extremal results in bootstrap percolation. The lower bound in~\cite{metastabilityeffects}, where the order of magnitude of the critical probability in $2$-neighbour bootstrap percolation on $[n]^2$ was determined, follows very easily from the fact that $G_{2,2}(n) = n$. In~\cite{bootstraphigh}, Balogh, Bollob{\'a}s, and Morris used the value of $G_{d,2}(n)$ for arbitrary $d$ as a vital tool to determine the critical probability in $2$-neighbour bootstrap percolation on high-dimensional grids. Finally, we remark that Balister, Bollob{\'a}s, Johnson, and Walters~\cite{randomMajority}, and Huang and Lee~\cite{deterministicHigh}, independently observed that $G_{d,d}(n) \leq c_d n^{d-1}$, where $c_d > 0$ is some constant depending on $d$ only, as infecting the boundary of $[n]^{d}$ (of size at most $2dn^{d-1}$) gives us a percolating set in the $d$-neighbour bootstrap process.

\section{Proof of the main result}
\label{sec:main}

In this section we prove Theorem~\ref{thm:pete}. The result $G_{d,d}(n)\geq n^{d-1}$ follows from our discussion of the monotonicity of the perimeter of the infected set. Therefore we need to prove that $G_{d,d}(n)\leq n^{d-1}$. Unlike for $d=1,2$, in the general case proving the upper bound turns out to be much more challenging.

\begin{proof}[Proof of Theorem~\ref{thm:pete}]
Let $G=[n]^d$ be the $d$-dimensional grid of size $n$. For $d\leq k\leq dn$, we define $V_k=\{v=(v_1,...,v_d)\in [n]^d:\sum_{i=1}^{d} v_i=k\}$. Note that $\bigcup_{k=d}^{dn} V_k=V([n]^d)$.

We will show that the set $A = \bigcup_{i=1}^{d} V_{in}$ percolates in $d$-neighbour bootstrap percolation on $[n]^d$. (We can immediately see that $|A| = n^{d-1}$ as for fixed values of $v_2,...,v_d$, there is exactly one choice of $v_1$ such that $v=(v_1,...,v_d) \in A$.) To do this, we will prove that, for all $1\leq s\leq d$, 
\begin{equation}
\label{eqn:F_s}
F_s=\bigcup_{i=1}^{n-1}V_{(s-1)n+i}\subseteq \langle V_{(s-1)n}\cup V_{sn}\rangle_d.
\end{equation}
Note that we have $V_0, \ldots, V_{d-1} =\emptyset$, and consequently the sets $V_n, \ldots , V_{\left(\lceil d/n \rceil -1 \right)n}$ are empty. First we deal with the ``bottom corner'' of the grid.

\begin{claim}
\label{claim:bottomCorner}
We have $\bigcup_{j=d}^{\lceil d/n \rceil n-1}V_j \subseteq \langle V_{\lceil d/n \rceil n} \rangle_d$.
\end{claim}
\begin{proof}
Any vertex $v \in V_{\lceil d/n \rceil n - 1}$ has $\sum_{i=1}^{d} v_i=\lceil d/n \rceil n - 1$. If there was some $1 \leq i \leq d$ such that $v_i = n$, then the vertex $v - (n-1)e_i$ would lie in $V_{( \lceil d/n \rceil -1 )n}$ which we know is empty, a contradiction. Hence, $v_i < n$ for all $1 \leq i \leq d$.

Therefore, again for all $1 \leq i \leq d$, we have $v+e_i \in V_{\lceil d/n \rceil n}$ infected. Therefore $v$ has at least $d$ infected neighbours and itself becomes infected. Since $v \in V_{\lceil d/n \rceil n-1}$ was arbitrary, all of $V_{\lceil d/n \rceil n-1}$ becomes infected. We proceed in this manner, in consecutive rounds infecting all vertices in $V_{\lceil d/n \rceil n-2}, V_{\lceil d/n \rceil n-3}, \ldots, V_d$. This completes the proof of the claim.
\end{proof}

Observe that $V_{dn} = \{(n,\ldots,n)\} \subset A$, hence we do not need to deal with the ``upper corner''. Therefore from now on we shall analyse the dynamics of the process ``sandwiched'' between two initially infected hyperplanes. Fix $1+\left \lceil \frac{d}{n}\right \rceil \leq s\leq d$ and assume that $V_{(s-1)n}\cup V_{sn}$ is infected. Given $v\in F_s$, let $t_v=\sum_{i=1}^{d}v_i-(s-1)n$. Next, for $v\in F_s$, we define
\begin{equation}
\label{eqn:Pre(v)}
\Pre(v)=\{v+e_j:v_j\leq t_v\}\cup \{v-e_j:v_j>t_v\}.
\end{equation}
For all $v\in F_s$, we have $|\Pre(v)|=d$. Therefore, if all vertices in $\Pre(v)$ are infected, then $v$ also becomes infected.

We define the \textit{infection witness tree} of $v$, $\IW(v)$, to be a directed labelled $d$-ary tree, with all edges directed away from the root and with vertices labelled with the elements of $F_s\cup V_{(s-1)n}\cup V_{sn}$, where these labels can be repeated in the tree. We construct $\IW(v)$ as follows. We start by declaring the root of the tree active and labelling it with $v$. Then, in consecutive rounds, we select an arbitrary active vertex. If the label $u$ of the vertex belongs to $V_{(s-1)n}\cup V_{sn}$, then this vertex becomes a leaf of $\IW(v)$ and we simply change its status to inactive. Otherwise, if the label $u$ of the vertex  belongs to $F_s$, then we attach $d$ active children to this vertex and label them with the elements of $\Pre(u)$. Then, we again declare the selected vertex inactive. See Figure \ref{fig:IW} for an example of a tree constructed in our algorithm.

\begin{figure}[htb]
\centering
\begin{tikzpicture}[nodes={draw}, ->]
\footnotesize
\node{422/3}
    child { node {322/2} 
        child { node {222/1}
            child { node [circle, draw]{221/0}
            edge from parent node[draw=none,left] {$3$}
            }
            child [missing] 
            child [missing]
        }
        child { node {332/3}
            child { node {432/4} 
                child { node [circle, draw]{532/5} }
                child { node [circle, draw]{442/5} }
                child { node [circle, draw]{433/5} }
            edge from parent node[draw=none,left] {$2$}
            }
            child [missing]
            child { node {333/4} 
                child [missing]
                child { node [circle, draw]{433/5} 
            edge from parent node[draw=none,left] {$3$}}
            }
            edge from parent node[draw=none,left] {$2$}
        }
    }
    child [missing]
    child [missing]
    child { node {432/4}
        child { node [circle, draw]{532/5} }
        child { node [circle, draw]{442/5} }
        child { node [circle, draw]{433/5} }
        edge from parent node[draw=none,left, below] {$2$}
    };
\end{tikzpicture}
\normalsize
\caption{The infection witness tree of $v=(4,2,2)$, $\IW((4,2,2))$, when $n=5$ and $d=3$. Here, the format of the vertex labels is $u / t_u$. The round vertices have labels in the initially infected set $A$. Observe that, as vertices are included in $A$ based only on the sum of their coordinates and not on the order of their values, any two vertices whose coordinates are a permutation of each other become infected simultaneously. Hence, for clarity, rather than drawing multiple children, we use edge labels to denote the number of children whose coordinates are a permutation of a given label (e.g., both $(4,3,2)$ and $(4,2,3)$ belong to $\Pre((4,2,2)))$.}
\label{fig:IW}
\end{figure}
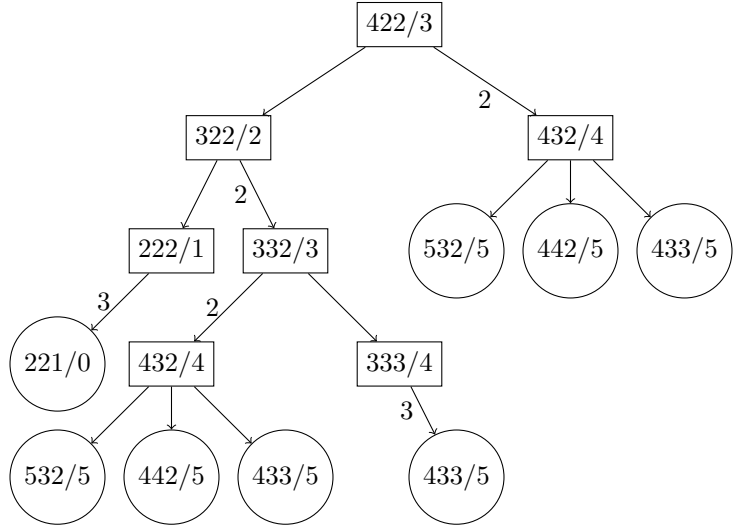

By definition, all leaves of the tree $\IW(v)$ are initially infected. Since $\IW(v)$ is a $d$-ary tree, if $\IW(v)$ is finite then $v$ becomes infected. Since every non-leaf of $\IW(v)$ belongs to $F_s$ which is a finite set, an infinite directed path in $\IW(v)$ would contain infinitely many instances of the same label. Hence, the finiteness of $\IW(v)$ follows immediately from the next lemma.

\begin{lem}
\label{lem:noCycles}
For any $v \in F_s$, $\IW(v)$ has no directed path $u^1,...,u^m,u^{m+1}=u^1$, where $u^{i+1}\in \Pre(u^i)$ for all $1\leq i\leq m$.
\end{lem}

\begin{proof}
Suppose for a contradiction that the lemma does not hold, so there exists a directed path $u^1,...,u^m,u^{m+1}=u^1$, where $u^{i+1}\in \Pre(u^i)$ for all $1\leq i\leq m$. Then, since the algorithm that we use to construct $\IW(v)$ is deterministic, we know that we have an infinite directed path with labels $(u^{i})_{i\geq 1}$, where $u^{i+1}\in \Pre(u^i)$ for all $i\geq 1$, and there is some $m \geq 1$ (in fact we could only have $m \geq 2$ even) such that $u^{i+m}=u^i$ for all $i\geq 1$. Hence, we can assume without loss of generality that $C=t_{u^1}=\max\limits_{1\leq i\leq m}t_{u^i}$.

As we traverse the directed path $(u^{i})_{i\geq 1}$, whenever $u^{i+1}\in \Pre(u^i)$ with $u^{i+1}=u^i+e_j$ for some $1\leq j\leq d$, by~\eqref{eqn:Pre(v)} we know that $u_j^i\leq t_{u^i}=t_{u^{i+1}}-1\leq t_{u^1}-1=C-1$. So we deduce that $u_j^{i+1}\leq C$. Now, given a vertex $v\in [n]^d$, we define
\[
 L_C(v) =\sum_{\substack{1\leq j\leq d:\\v_j\geq C+1}}v_j.
\]
(I.e., $L_C(v)$ is the sum of all coordinates of $v$ that are larger than $C$.) Therefore, when $t_{u^{i+1}}>t_{u^i}$, we know that
\[
 L_C(u^{i+1})=\sum_{\substack{1\leq j\leq d:\\u_j^{i+1}\geq C+1}}u_j^{i+1}=\sum_{\substack{1\leq j\leq d:\\u_j^i\geq C+1}}u_j^i=L_C(u^i).
\]

However, if $u^{i+1}\in \Pre(u^i)$ with $u^{i+1}=u^i-e_j$ for some $1\leq j\leq d$, then it is clear that $L_C(u^{i+1})\leq L_C(u^i)$. Additionally, by the maximality of $t_{u^1}$, we have that $t_{u^2}=t_{u^1}-1$. So we deduce that $u^2=u^1-e_j$ for some $1\leq j\leq d$. Then, by~\eqref{eqn:Pre(v)}, we have $u_j^1\geq t_{u^1}+1=C+1$. Hence, as $u_j^2=u_j^1-1$, we clearly have $L_C(u^2)<L_C(u^1)$. Thus, following from the fact that $L_C$ never increases as we go along our directed path, $L_C(u^{m+1})<L_C(u^1)$. This implies that $u^{m+1}\neq u^1$, a contradiction to our previous assumption.

Hence, $\IW(v)$ has no directed paths on which the same label is repeated more than once and, as discussed earlier, the whole tree is finite.
\end{proof}

The following corollary is immediate, and concludes the proof of Theorem~\ref{thm:pete}.
\begin{cor}
For any vertex $v\in F_s$, $\IW(v)$ is finite. Consequently, $v$ becomes infected in finite time and, since $v\in F_s$ was chosen arbitrarily, all of $F_s$ becomes infected.
\end{cor}
\end{proof}

\section{Percolation time}
\label{sec:time}

In this section, we exploit the machinery developed in Section~\ref{sec:main} to prove Theorem~\ref{thm:smallestTime}. In particular, by tightening our analysis of the height of $\IW(v)$, we will show that the bootstrap percolation process started from the set $A$, defined in~\eqref{eqn:initialSet}, terminates after at most $(d+2)n^2+n$ time steps.

\begin{proof}[Proof of Theorem~\ref{thm:smallestTime}]
The case $d=1$ is trivial; to minimise the percolation time we simply place one infected vertex at $\lceil n/2 \rceil$.

The case $d=2$ is an interesting puzzle. As for the upper bound on $m_2(n)$, we can clearly see that a diagonal percolates $[n]^2$ in $n-1$ steps. For the lower bound, we observe that, by the perimeter argument, at least one of the following two neighbouring vertices: $( \lceil n/2 \rceil, \lceil n/2 \rceil )$ and $( \lfloor n/2 \rfloor, \lceil (n+1)/2 \rceil )$, must be initially healthy. (For $n$ even these two vertices are neighbours in the central $2 \times 2$ subsquare, while for $n$ odd the former one is in the very centre of the grid, with the latter one being its neighbour on the left.) Now, we keep applying the perimeter argument: every time a vertex becomes infected, it must be by virtue of exactly $2$ infected neighbours. Moreover, it is an immediate observation that the perimeter of the infected set would also decrease if two neighbouring vertices became infected at the same time step. Hence, only the corner vertices can become infected without having any of their neighbours still healthy after their infection. This means that, for any percolating set of size $n$, we can construct a path of neighbouring vertices, starting at either $( \lceil n/2 \rceil, \lceil n/2 \rceil )$ or $( \lfloor n/2 \rfloor, \lceil (n+1)/2 \rceil )$ and finishing in one of the corners of the grid, such that the consecutive vertices of the path become infected at strictly later time steps. All such paths have length at least $n-1$: for $n$ even we could take a path from $( \lceil n/2 \rceil, \lceil n/2 \rceil )$ to $(1,1)$, while for $n$ odd from $( \lfloor n/2 \rfloor, \lceil (n+1)/2 \rceil )$ to $(1,1)$. This gives us the desired lower bound on $m_2(n)$. (We remark that Benevides and Przykucki~\cite{maxtimeMinsize} showed that the maximum percolation time for a set of size $n$ in $[n]^2$ is equal to the integer nearest to $(5n^2 - 2n)/8$).

Hence, let us assume that $d \geq 3$. Here, the lower bound follows by an identical argument to the one we used for $d=2$. Consider the vertices
\[
(\lceil n/2 \rceil, \lceil n/2 \rceil, \ldots, \lceil n/2 \rceil) \mbox{ and } (\lceil n/2 \rceil+1, \lceil n/2 \rceil,  \ldots, \lceil n/2 \rceil),
\]
and observe that at least one of them has to be initialy healthy by the perimeter argument. Then, every path from one of these vertices to a corner of the grid has length $dn/2+O(1)$, meaning that $m_d(n) \geq dn/2+O(1)$ as claimed. The upper bound on $m_d(n)$ in Theorem~\ref{thm:smallestTime} follows immediately from the next lemma, which sharpens the analysis in Lemma~\ref{lem:noCycles}.

\begin{lem}
\label{lem:quadraticHeight}
Let $v \in F_s$ and let $u^1 = v,u_2,\ldots,u^m$ be a directed path in $\IW(v)$, with $u^{i}\in \Pre(u^{i-1})$ for all $2\leq i\leq m$. Then $m \leq (d+2)n^2+n+1$.
\end{lem}

\begin{proof}
Given $u \in F_s$, let $h(u) = \sum_{i=1}^d u_i^2$ be the sum of squares of the coordinates of $u$. The idea of the proof is to show that long paths in $\IW(v)$, corresponding to large values of $m$, result in very small values of $h$; we want to show that $m > (d+2)n^2+n+1$ would give $h(u^m) < 0$, which is a clear contradiction.

For notational convenience, we shall denote $t_i = t_{u^i}$. Clearly $|t_m-t_1 | \leq n$, since $u^1 \in F_s$, and $ u^m \in F_s \cup V_{(s-1)n} \cup V_{sn}$. Thus, we can find a subset $I \subset \{2,3,\ldots,m\}$ with $|I| \geq m-1-n$ and $|I|$ even, such that we can group the elements of $I$ into pairs
\[
 (i^1, j^1), (i^2, j^2), \ldots, (i^{|I|/2}, j^{|I|/2}),
\]
with the following property: for all $1 \leq k \leq |I|/2$, we have
\begin{enumerate}
 \item $t_{i^k} =  t_{j^k} - 1$,
 \item $t_{i^k} =  t_{i^k-1} - 1$ and $t_{j^k} =  t_{j^k-1} + 1$.
\end{enumerate}
In other words, all but at most $n$ elements of the subpath $u^2,\ldots,u^m$ can be partitioned into pairs $(u^i, u^j)$ such that $u^i$ lies one level below $u^{i-1}$, as well as one level below $u^j$, which in turn lies one level above $u^{j-1}$ (where the level of a vertex $u$ is equal to $t_u$, see Figure~\ref{fig:longPath}).

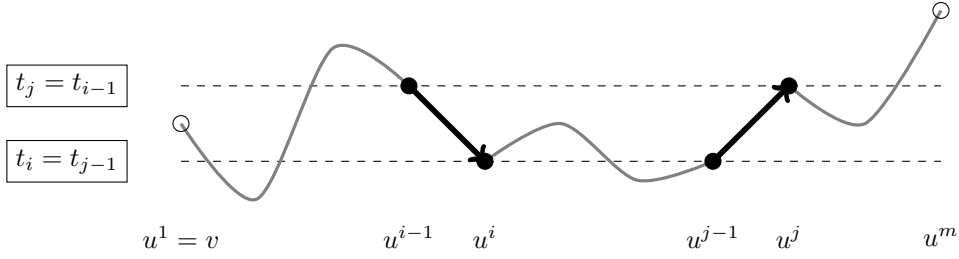
\begin{figure}[htb]
\centering
\begin{tikzpicture}
\footnotesize
 \draw[dashed] (0,1) -- (10,1);
 \draw[dashed] (0,0) -- (10,0);

\draw [black, very thick, gray] plot [smooth] coordinates {(0,0.5) (1,-0.5) (2,1.5) (3,1) };

\draw [black, very thick, gray] plot [smooth] coordinates {(4,0) (5,0.5) (6,-0.25) (7,0) };

\draw [black, very thick, gray] plot [smooth] coordinates {(8,1) (9,0.5) (10,2) };

 \draw[line width = 0.75mm,->] (3,1) -- (4,0);
 \draw[line width = 0.75mm,->] (7,0) -- (8,1);
 
\node[draw] at (-1.5,0) {$t_{i} = t_{j-1}$};
\node[draw] at (-1.5,1) {$t_{j} = t_{i-1}$};

\filldraw (3,1) circle (3pt);
\filldraw (4,0) circle (3pt);
\filldraw (7,0) circle (3pt);
\filldraw (8,1) circle (3pt);
\draw (0,0.5) circle (3pt);
\draw (10,2) circle (3pt);
 
\node at (0,-1) {$u^{1} = v$};
\node at (3,-1) {$u^{i-1}$};
\node at (4,-1) {$u^{i}$};
\node at (7,-1) {$u^{j-1}$};
\node at (8,-1) {$u^{j}$};
\node at (10,-1) {$u^{m}$};

\end{tikzpicture}
\normalsize
\caption{Schematic depiction of the level of consecutive vertices on a directed path $u^1 = v, u_2, \ldots,u^m$, with $u^i$ and $u^j$ paired through $(i,j) \in I$.}
\label{fig:longPath}
\end{figure}

By a reasoning analogous to the one in the proof of Lemma~\ref{lem:noCycles} and by the convexity of $x^2$, we have
\[
h(u^{i-1}) - h(u^i) \geq (t_{i-1}+1)^2 - t_{i-1}^2 = 2 t_{i-1} + 1.
\]
On the other hand, we have
\[
h(u^{j}) - h(u^{j-1}) \leq (t_{j-1}+1)^2 - (t_{j-1})^2 = 2 t_{j-1} + 1.
\]
However, by the properties of our pairs, we have
\[
2 t_{i-1} + 1 = 2 (t_{i}+1) + 1 = 2 t_{j} + 1 = 2 (t_{j-1}+1) + 1 = 2 t_{j-1} + 3.
\]
Hence, the sum of the changes in the value of $h$ as we move from $u^{i-1}$ to $u^i$, and from $u^{j-1}$ to $u^j$ (in an arbitrary order), is at most $-2$. We clearly have $h(u^1) \leq dn^2$, and through the at most $n$ unpaired moves we increase the value of $h$ by at most $n(n^2-(n-1)^2) < 2n^2$. Therefore we must have $|I|/2 \leq (d+2)n^2/2$, which gives $m \leq (d+2)n^2+n+1$. This concludes the proof of the lemma.
\end{proof}

Theorem~\ref{thm:smallestTime} now follows immediately, as the label $u$ of any vertex of $\IW(v)$ becomes infected at most one step after all the vertices in $\Pre(u)$ are infected. The height of $\IW(v)$, being bounded by $ (d+2)n^2+n$, implies the desired bound on the percolation time.
\end{proof}

In fact, numerical simulations suggest that, for $d \geq 3$, the percolation time of the process started from $A$ (as given in~\eqref{eqn:initialSet}) grows quadratically in $n$. For $d=3$ the process terminates after $n^2/2-n+O(1)$ time steps, for $d=4$ it lasts $2n^2/3-2n/3+O(1)$ steps, and for $d=5$ infection takes $n^2-3n+O(1)$ steps. We do not believe that these exact sets $A$ minimise percolation time of a set of size $n^{d-1}$ in $[n]^d$; for example, taking initially infected sets $A' = A'_d = \bigcup_{i=1}^{d} V_{in - \lfloor n/2 \rfloor}$ appears to lead to strictly smaller coefficients of $n^2$. However, motivated by Theorem~\ref{thm:smallestTime} and the results of our simulations, we expect the answer to the following question to be positive.

\begin{qu}
Is $m_d(n) = \Theta(n^2)$ for all $d \geq 3$?
\end{qu}

One could also ask about $m(\bT_n^d)$, the size of the smallest percolating sets in $d$-neighbour bootstrap percolation on $\bT_n^d$, the $d$-dimensional torus of size $n$. It is known that $m(\bT_n^2) = n-1$, but the situation quickly becomes more complicated in higher dimensions. Our result immediately implies that $m(\bT_n^d) \leq n^{d-1}$, but this bound is not sharp. For example, for $d=3$ we could infect an $[n-1]^3$ cube using $(n-1)^2$ initially infected vertices, and then use the boundary conditions of the torus to infect
\[
([n] \times [n-1] \times [n-1]) \cup ([n-1] \times [n] \times [n-1]) \cup ([n-1] \times [n-1] \times [n])
\]
with only three additional initially infected vertices. It is easy to see that this set percolates the torus, giving us
\[
m(\bT_n^3) \leq (n-1)^{2}+3 = n^2-2n+4 < n^2
\]
for all $n \geq 3$.

\begin{qu}
What is the value of $m(\bT_n^d)$ for $d \geq 3$?
\end{qu}

\vspace{0.5em}
\noindent \textbf{Acknowledgement} We would like to thank Gabor Pete for helpful guidance concerning earlier results on the smallest percolating sets, and Ellen Harrison for help in preparing this manuscript.

\bibliographystyle{plain}

\end{document}